\documentclass[a4paper,11pt]{amsart}

\usepackage{a4wide}
\usepackage{tikz}
\usepackage{cite}

\newcommand{\T}{\mathcal{T}}

\newif\ifdetails
\detailstrue
\newcommand{\DETAIL}[1]%
{\ifdetails\par\fbox{\begin{minipage}{0.9\linewidth}\emph{Detail:}
      #1\end{minipage}}\par\fi}
\newcommand{\TODO}[1]%
{\ifdetails\par\fbox{\begin{minipage}{0.9\linewidth}\textbf{TODO:}
      #1\end{minipage}}\par\fi}

\usepackage{makecell}
\usepackage{relsize}

\newtheorem{lemma}{Lemma}
\newtheorem{proposition}[lemma]{Proposition}
\newtheorem{theorem}[lemma]{Theorem}
\newtheorem{corollary}[lemma]{Corollary}
\theoremstyle{remark}

\newtheorem*{conjecture}{Conjecture}

\usepackage{caption}
\usepackage{subcaption}
\usepackage{float} 
\usepackage[pdftex,a4paper,
citecolor = blue, colorlinks=true,urlcolor=blue]{hyperref}
\newcommand\Tstrut{\rule{0pt}{2.6ex}}         
\newcommand\Bstrut{\rule[-0.9ex]{0pt}{0pt}}  
\newcommand{\old}[1]{{}}

\title{Further results on the inducibility of $d$-ary trees}

\author{Audace A. V. Dossou-Olory}
\author{Stephan Wagner}
\thanks{The first author was supported by Stellenbosch University in association with African Institute for Mathematical Sciences (AIMS) South Africa; the second author was supported by the National Research Foundation of South Africa, grant number 96236.}
\address{Audace A. V. Dossou-Olory and Stephan Wagner\\ Department of Mathematical Sciences \\ Stellenbosch University \\ Private Bag X1, Matieland 7602 \\ South Africa}
\email{audace@aims.ac.za, swagner@sun.ac.za}

\begin{document}

\setlength{\abovedisplayskip}{3pt}
\setlength{\belowdisplayskip}{3pt}

\begin{abstract}
A subset of leaves of a rooted tree induces a new tree in a natural way. The density of a tree $D$ inside a larger tree $T$ is the proportion of such leaf-induced subtrees in $T$ that are isomorphic to $D$ among all those with the same number of leaves as $D$. The inducibility of $D$ measures how large this density can be as the size of $T$ tends to infinity. In this paper, we explicitly determine the inducibility in some previously unknown cases and find general upper and lower bounds, in particular in the case where $D$ is balanced, i.e., when its branches have at least almost the same size. Moreover, we prove a result on the speed of convergence of the maximum density of $D$ in strictly $d$-ary trees $T$ (trees where every internal vertex has precisely $d$ children) of a given size $n$ to the inducibility as $n \to \infty$, which supports an open conjecture.
\end{abstract}

\maketitle

\section{Introduction and statement of results}\label{Sec:intro}

The inducibility is a recently introduced invariant that captures how often a fixed rooted tree can occur ``inside'' a large rooted tree. For its formal definition, let us start with some fundamental terminology and notation. By a \emph{$d$-ary tree}, we mean a rooted tree whose internal (non-leaf) vertices all have at least two and at most $d$ children. In the special cases $d=2$ and $d=3$, we speak of \emph{binary} and \emph{ternary} trees, respectively. A rooted tree is called \emph{strictly $d$-ary} if every internal vertex has exactly $d$ children. For our purposes, it is natural to measure the size of a rooted tree $T$ by the number of leaves, which we denote by $\|T\|$. A subset $S$ of leaves of a $d$-ary tree induces another $d$-ary tree in a natural way: we first take the smallest subtree of $T$ that contains all the leaves in $S$, and then repeatedly suppress all vertices with only one child by contracting the two adjacent edges to a single edge, until no vertex with a single child remains. The procedure is illustrated in Figure~\ref{leaf-induced}. The resulting tree is called a \emph{leaf-induced subtree}. Its root is precisely the most recent common ancestor of the leaves in $S$. A \emph{copy} of a fixed $d$-ary tree $D$ is any leaf-induced subtree that is isomorphic to $D$; we denote the number of distinct copies of $D$ in $T$ by $c(D,T)$. In other words, $c(D,T)$ is the number of sets of $\|D\|$ leaves of $T$ that induce a tree isomorphic to $D$. A normalised version of this quantity is the \emph{density} $\gamma(D,T)$, which is defined as
$$\gamma(D,T) = \frac{c(D,T)}{\binom{\|T\|}{\|D\|}}\,.$$
It can be seen as the probability that a randomly chosen leaf subset of $\|T\|$ induces a subtree isomorphic to $D$, and therefore always lies between $0$ and $1$. Now we finally define the $d$-ary \emph{inducibility} of a fixed $d$-ary tree $D$ as the limit superior of the density as the size of $T$ tends to infinity:
$$I_d(D) = \limsup_{\substack{\|T\| \to \infty \\ T\ d\text{-ary tree}}} \gamma(D,T) = \limsup_{n \to \infty} \max_{\substack{\|T\| = n \\ T\ d\text{-ary tree}}} \gamma(D,T)\,.$$
It is a nontrivial fact, proven in \cite{AudaceStephanPaper1}, that one can replace $\limsup$ by an ordinary limit in the second expression:
\begin{equation}\label{eq:limit}
I_d(D) = \lim_{n \to \infty} \max_{\substack{\|T\| = n \\ T\ d\text{-ary tree}}} \gamma(D,T)\,.
\end{equation}
Moreover, it is possible to restrict $T$ to strictly $d$-ary trees:
$$I_d(D) = \lim_{n \to \infty} \max_{\substack{\|T\| = (d-1)n+1 \\ T \text{ strictly } d\text{-ary tree}}} \gamma(D,T)\,.$$
Note here that $\|T\| \equiv 1 \mod (d-1)$ for every strictly $d$-ary tree $T$ (which is well known and easy to show), hence the restriction to values of the form $(d-1)n+1$.

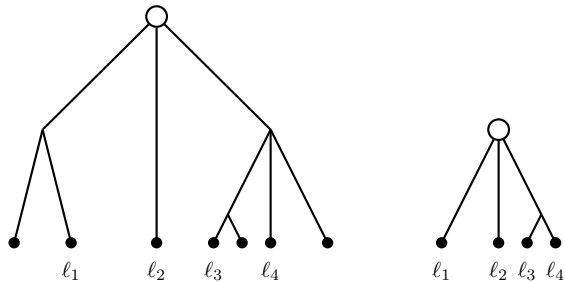
\begin{figure}[!h]\centering  

\begin{tikzpicture}[thick,scale=0.75, transform shape]
\node [circle,draw] (r) at (0,0) {};

\draw (r) -- (-2,-2);
\draw (r) -- (0,-4);
\draw (r) -- (2,-2);
\draw (-2,-2) -- (-2.5,-4);
\draw (-2,-2) -- (-1.5,-4);
\draw (2,-2) -- (1,-4);
\draw (2,-2) -- (2,-4);
\draw (2,-2) -- (3,-4);
\draw (1.25,-3.5) -- (1.5,-4);

\node [fill,circle, inner sep = 2pt ] at (-2.5,-4) {};
\node [fill,circle, inner sep = 2pt ] at (-1.5,-4) {};
\node [fill,circle, inner sep = 2pt ] at (0,-4) {};
\node [fill,circle, inner sep = 2pt ] at (1,-4) {};
\node [fill,circle, inner sep = 2pt ] at (1.5,-4) {};
\node [fill,circle, inner sep = 2pt ] at (2,-4) {};
\node [fill,circle, inner sep = 2pt ] at (3,-4) {};

\node at (-1.5,-4.5) {$\ell_1$};
\node at (0,-4.5) {$\ell_2$};
\node at (1,-4.5) {$\ell_3$};
\node at (2,-4.5) {$\ell_4$};

\node [circle,draw] (r1) at (6,-2) {};

\draw (r1) -- (5,-4);
\draw (r1) -- (6,-4);
\draw (r1) -- (7,-4);
\draw (6.75,-3.5)--(6.5,-4);

\node [fill,circle, inner sep = 2pt ] at (5,-4) {};
\node [fill,circle, inner sep = 2pt ] at (6,-4) {};
\node [fill,circle, inner sep = 2pt ] at (6.5,-4) {};
\node [fill,circle, inner sep = 2pt ] at (7,-4) {};

\node at (5,-4.5) {$\ell_1$};
\node at (6,-4.5) {$\ell_2$};
\node at (6.5,-4.5) {$\ell_3$};
\node at (7,-4.5) {$\ell_4$};
\end{tikzpicture}
\caption{A ternary tree (left) and the subtree induced by four leaves $\{\ell_1,\ell_2,\ell_3,\ell_4\}$ (right).}\label{leaf-induced}
\end{figure}

The definition described above parallels the notion of inducibility of graphs, which is defined in an analogous way (``copies'' being isomorphic embeddings). Its investigation began with a paper by Pippenger and Golumbic \cite{pippenger1975inducibility}, and there is a substantial amount of literature on this parameter (see \cite{balogh2016maximum,hatami2014inducibility,hirst2014inducibility,even2015note} for some recent examples).

Inducibility in the context of trees, on the other hand, was first mentioned by Bubeck and Linial in 2016 \cite{bubeck2016local}, though ``copies'' were also defined as isomorphic embeddings there, and the proportion was taken among subtrees of the same size, which means that the denominator depends on the whole tree rather than just its size.

\medskip

The definition in terms of leaf-induced subtrees goes back to a recent article by Czabarka, Sz{\'e}kely and the second author of this paper \cite{czabarka2016inducibility}, originally only in the binary case. It was motivated by a concrete question from phylogenetics involving a structure known as a \emph{tanglegram} that consists of two binary trees whose leaves are connected by a perfect matching. 

\medskip

Not a lot is known about the $d$-ary inducibility; in particular, its exact value has only been determined in very few cases. Even though the limit in \eqref{eq:limit} can be shown to exist, it is not easy to evaluate. It is therefore desirable to have ways to bound the inducibility from above or below. The maximum value of $I_d(D)$ is equal to $1$, and it is attained precisely when $D$ is a \emph{binary caterpillar}, i.e., a binary tree whose internal vertices form a path rooted at one of its ends (see ~\cite{AudaceStephanPaper1}). Note that this is trivial if $\|D\| = 1$ or $\|D\| = 2$, since there are no other trees of the same size in these cases. On the other hand, the following lower bound was shown in~\cite{AudaceStephanPaper1} to hold for all $d$-ary trees with $k > 1$ leaves:
\begin{align*}
I_d(D)\geq \frac{(k-1)!}{k^{k-1}-1}\,,
\end{align*}
with equality for the star when $k=d$. In particular, the $d$-ary inducibility is always positive. The minimum of $I_d(D)$ (given $\|D\|$) is not known in general, though. In this context, it is worth mentioning that the quantity
\begin{align*}
\liminf_{\substack{\|T\|\to \infty \\ T~\text{$d$-ary tree}}} \gamma(D,T)\,,
\end{align*}
which is the minimum analogue of the inducibility, is much better understood. Specifically, it was shown in~\cite{MinDensity} that this quantity is always equal to $0$ unless $D$ is a binary caterpillar, in which case an explicit formula can be given.

\medskip

While the limit in~\eqref{eq:limit} is difficult to evaluate, it can be used to approximate $I_d(D)$; this method was applied in~\cite{SmallTrees} to two concrete examples. For this purpose, information on the speed of convergence is crucial. It was shown in~\cite{AudaceStephanPaper1} that
\begin{equation}\label{eq:arb_asymp}
I_d(D) \leq \max_{\substack{\|T\| = n \\ T\ d\text{-ary tree}}} \gamma(D,T) \leq I_d(D) + \frac{\|D\|(\|D\|-1)}{n}\,,
\end{equation}
so the sequence of maximum densities converges to the limit with an error term of order at most $\mathcal{O}(n^{-1})$. There are concrete examples (see~\cite{czabarka2016inducibility}) showing that the order of magnitude of this error term cannot be improved in general. In the case where only strictly $d$-ary trees are considered, it was shown in~\cite{AudaceStephanPaper1} that
\begin{equation}\label{eq:strict_asymp}
\max_{\substack{\|T\| = (d-1)n+1 \\ T \text{ strictly } d\text{-ary tree}}} \gamma(D,T) = I_d(D) + \mathcal{O}(n^{-1/2})\,,
\end{equation}
but it is an open question whether this error term can be improved.

\medskip

This paper contributes to all the aforementioned questions. In particular, we determine the precise value of the inducibility in some new cases. Among the trees to which our method applies are what we call \emph{even} $d$-ary trees (based on the property that the leaves in these trees are ``evenly distributed''). This extends previous results in the binary case \cite{czabarka2016inducibility}. 
The even $d$-ary tree with $k$ leaves, denoted $E_k^d$, is defined recursively as follows: 
\begin{itemize}
\item For $k \leq d$, $E_k^d$ is a star, consisting only of the root and $k$ leaves.
\item If $k > d$, we express it as $k = ds + b$, with $b \in \{0,1,\ldots,d-1\}$. Take $d-b$ copies of $E_s^d$ and $b$ copies of $E_{s+1}^d$, and connect a new common root to each of their roots by an edge to obtain $E_k^d$.
\end{itemize}
Figure~\ref{Someeventernarytrees} shows the even ternary trees with up to nine leaves.
We have the following theorem:

\begin{theorem}\label{IndEvendAryTree}
Let the constants $c_k$ be defined recursively by $c_0 = c_1 = 1$, and by
$$c_{ds+b}=\binom{d}{b} \frac{c_s^{d-b}\cdot c_{s+1}^b}{d^{ds+b}-d}$$
for every $s\geq 0$ and every $ b \in\{0,1,\ldots, d-1\}$.
The $d$-ary inducibility of the even tree $E_k^d$ is given by $I_d(E_k^d) = k! c_k$.
\end{theorem}

As an example, Table~\ref{tab:eventernarytrees} indicates the (ternary) inducibilities of the first few even ternary trees.
\begin{table}[htbp]\centering
\caption{Some values of $I_3(E^3_k)$.}\label{tab:eventernarytrees}
\begin{tabular}{l||l|l|l|l|l|l|l|l|l|l|l|l}
$k$  & 1 & 2 & 3 & 4 & 5 & 6 & 7 & 8 & 9 & 10 & 11 & 12 \\  \hline  
$I_3(E^3_k)$ & $1$ & $1$ & $
 \frac{1}{4}$ & $
\frac{6}{13}$ & $
 \frac{3}{8}$ & $
 \frac{15}{121}$ & $ \frac{15}{208}$ & $ 
\frac{35}{2186}$ & $
 \frac{7}{5248}$ & $
 \frac{1575}{255886}$ & $
 \frac{4725}{453596}$ & $
 \frac{1247400}{194594881}$  \Tstrut\Bstrut \Tstrut\Bstrut \\ 
\end{tabular}
\end{table} 

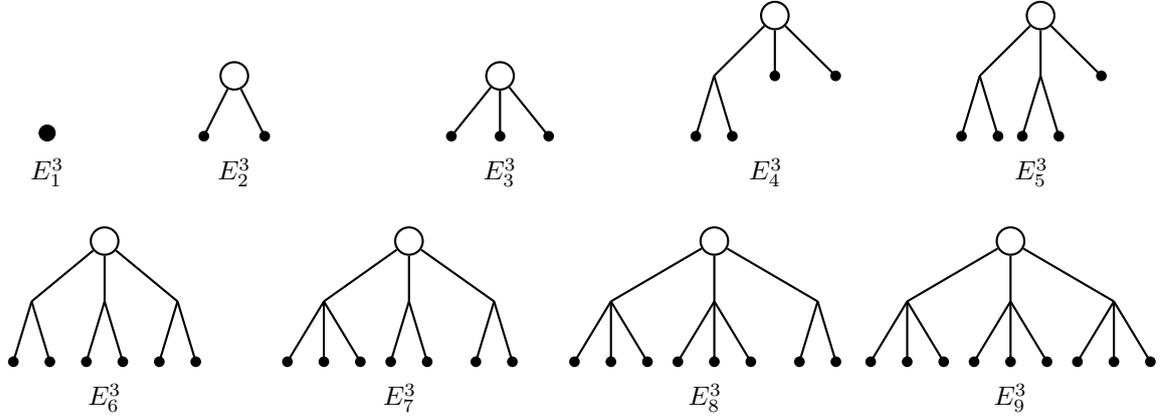
\begin{figure}[htbp]\centering
\begin{subfigure}[b]{0.1\textwidth} \centering 
\begin{tikzpicture}
\filldraw [black] circle (3pt);
\end{tikzpicture}
\caption{$E^3_1$}
  \end{subfigure}
\begin{subfigure}[b]{0.2\textwidth} \centering 
\begin{tikzpicture}[thick,level distance=10mm, scale=0.8]
\tikzstyle{level 1}=[sibling distance=10mm]
\tikzstyle{level 2}=[sibling distance=6mm]
\node [circle,draw]{}
child {[fill] circle (2pt)}
child {[fill] circle (2pt)};
\end{tikzpicture}
\caption{$E^3_2$}
  \end{subfigure}\quad
\begin{subfigure}[b]{0.2\textwidth} \centering 
\begin{tikzpicture}[thick,level distance=10mm, scale=0.8]
\tikzstyle{level 1}=[sibling distance=8mm]
\node [circle,draw]{}
child {[fill] circle (2pt)}
child {[fill] circle (2pt)}
child {[fill] circle (2pt)};
\end{tikzpicture}
\caption{$E^3_3$}
  \end{subfigure}\quad
\begin{subfigure}[b]{0.2\textwidth} \centering 
\begin{tikzpicture}[thick,level distance=10mm, scale=0.8]
\tikzstyle{level 1}=[sibling distance=10mm]
\tikzstyle{level 2}=[sibling distance=6mm]
\node [circle,draw]{}
child {child {[fill] circle (2pt)}child {[fill] circle (2pt)}}
child {[fill] circle (2pt)}
child {[fill] circle (2pt)};
\end{tikzpicture}
\caption{$E^3_4$}
  \end{subfigure}\quad
\begin{subfigure}[b]{0.2\textwidth} \centering 
\begin{tikzpicture}[thick,level distance=10mm, scale=0.8]
\tikzstyle{level 1}=[sibling distance=10mm]
\tikzstyle{level 2}=[sibling distance=6mm]
\node [circle,draw]{}
child {child {[fill] circle (2pt)}child {[fill] circle (2pt)}}
child {child {[fill] circle (2pt)}child {[fill] circle (2pt)}}
child {[fill] circle (2pt)};
\end{tikzpicture}
\caption{$E^3_5$}
  \end{subfigure}\newline \newline
\begin{subfigure}[b]{0.2\textwidth} \centering 
\begin{tikzpicture}[thick,level distance=10mm, scale=0.8]
\tikzstyle{level 1}=[sibling distance=12mm]
\tikzstyle{level 2}=[sibling distance=6mm]
\node [circle,draw]{}
child {child {[fill] circle (2pt)}child {[fill] circle (2pt)}}
child {child {[fill] circle (2pt)}child {[fill] circle (2pt)}}
child {child {[fill] circle (2pt)}child {[fill] circle (2pt)}};
\end{tikzpicture}
\caption{$E^3_6$}
  \end{subfigure}\quad
\begin{subfigure}[b]{0.25\textwidth} \centering 
\begin{tikzpicture}[thick,level distance=10mm, scale=0.8]
\tikzstyle{level 1}=[sibling distance=14mm]
\tikzstyle{level 2}=[sibling distance=6mm]
\node [circle,draw]{}
child {child {[fill] circle (2pt)}child {[fill] circle (2pt)}child {[fill] circle (2pt)}}
child {child {[fill] circle (2pt)}child {[fill] circle (2pt)}}
child {child {[fill] circle (2pt)}child {[fill] circle (2pt)}};
\end{tikzpicture}
\caption{$E^3_7$}
  \end{subfigure}\qquad ~
\begin{subfigure}[b]{0.25\textwidth} \centering 
\begin{tikzpicture}[thick,level distance=10mm, scale=0.8]
\tikzstyle{level 1}=[sibling distance=17mm]
\tikzstyle{level 2}=[sibling distance=6mm]
\node [circle,draw]{}
child {child {[fill] circle (2pt)}child {[fill] circle (2pt)}child {[fill] circle (2pt)}}
child {child {[fill] circle (2pt)}child {[fill] circle (2pt)}child {[fill] circle (2pt)}}
child {child {[fill] circle (2pt)}child {[fill] circle (2pt)}};
\end{tikzpicture}
\caption{$E^3_8$}
  \end{subfigure}
\begin{subfigure}[b]{0.25\textwidth} \centering 
\begin{tikzpicture}[thick,level distance=10mm, scale=0.8]
\tikzstyle{level 1}=[sibling distance=17mm]
\tikzstyle{level 2}=[sibling distance=6mm]
\node [circle,draw]{}
child {child {[fill] circle (2pt)}child {[fill] circle (2pt)}child {[fill] circle (2pt)}}
child {child {[fill] circle (2pt)}child {[fill] circle (2pt)}child {[fill] circle (2pt)}}
child {child {[fill] circle (2pt)}child {[fill] circle (2pt)}child {[fill] circle (2pt)}};
\end{tikzpicture}
\caption{$E^3_9$}
  \end{subfigure}
\caption{All even ternary trees with at most nine leaves.}\label{Someeventernarytrees}
\end{figure}

In addition, we show that the asymptotic formula
$$\lim_{n \to \infty} \max_{\substack{\|T\| = (d-1)n+1 \\ T \text{ strictly } d\text{-ary tree}}} \gamma(E_k^d,T) = I_d(E_k^d) + \mathcal{O}(n^{-1})$$
holds for all even trees, lending support to the conjecture that the error term in~\eqref{eq:strict_asymp} can generally be improved to $\mathcal{O}(n^{-1})$.

\medskip

Theorem~\ref{IndEvendAryTree} will be proven as part of a general approach in which a strict version of even trees plays a major role. For the inducibility of arbitrary $d$-ary trees, our approach yields both a general lower bound (Theorem~\ref{Completeproportion}) and an upper bound (Proposition~\ref{ind less prod ind}). In both cases, the bounds are determined recursively by decomposing a rooted tree into its \emph{branches}, i.e., the smaller trees that result as components when the root is removed. As it turns out, the inducibility of a tree can often be bounded in terms of the product of the inducibilities of its branches, both from above and below. As a particularly simple example, we have
$$I_d(D)\leq \prod_{i=1}^d I_d(D_i)$$
for every $d$-ary tree $D$ with branches $D_1,D_2,\ldots,D_d$ in which branches with the same number of leaves are isomorphic (Corollary~\ref{indA}). We provide a more general version of this inequality as well as further upper and lower bounds of a similar nature. We will demonstrate in some examples how they are applied to compute or approximate the inducibility in different cases.

\section{A special limit}

The proof of Theorem~\ref{IndEvendAryTree} (as well as other results) relies on a general recursion for the number of copies $c(D,T)$ of a tree $D$ inside a larger tree $T$, based on the decomposition of a rooted tree into its branches. It will be useful for notational purposes to allow empty trees (in particular, as branches of a tree) with $0$ leaves. If $D$ is empty, then we set $c(D,T) = 1$; accordingly, we will also set $I_d(D) = 1$ if $D$ is empty. If $T$ is empty, but $D$ is not, then $c(D,T) = 0$.

For a $d$-ary tree $D$ with branches $D_1,D_2,\ldots,D_d$ (some of which are allowed to be empty), we define the equivalence relation $\sim_{D}$ on the set of all permutations of $[d] = \{1,2,\ldots,d\}$ as follows: for two permutations $\pi$ and $\pi^{\prime}$ of $[d]$,
$$\big(\pi(1),\pi(2),\ldots,\pi(d)\big) \sim_{D} \big(\pi^{\prime}(1),\pi^{\prime}(2),\ldots,\pi^{\prime}(d)\big)$$
if for every $j \in [d]$, the tree $D_{\pi(j)}$ is isomorphic to $D_{\pi^{\prime}(j)}$ as a rooted tree (i.e., there is a root-preserving isomorphism between the two; two empty trees are of course considered isomorphic). Moreover, let $M(D)$ be a complete set of representatives of all equivalence classes of $\sim_{D}$. One verifies easily that all these equivalence classes have the same cardinality.
The following identity holds for every $d$-ary tree $T$ with branches $T_1,T_2,\ldots,T_d$ (some of which might also be empty).
\begin{equation}\label{The general recursion} 
c(D,T)=\sum_{i=1}^d c(D,T_i)+ \sum_{\pi \in M(D)}~\prod_{j=1}^d c\big(D_{\pi(j)},T_{j}\big)\,.
\end{equation}
Formula~\eqref{The general recursion} is established as follows:
\begin{itemize}
\item The term $\sum_{i=1}^d c(D,T_i)$ is the number of subsets of leaves that belong to a single branch of $T$ and induce a copy of $D$.
\item The expression $\prod_{j=1}^d c\big(D_{\pi(j)},T_{j}\big)$ stands for the number of copies of $D$ in which its branches $D_{\pi(1)},D_{\pi(2)},\ldots, D_{\pi(d)}$ are induced by subsets of leaves of $T_{1},T_{2}, \ldots,T_{d}$, respectively. Note that it is consistent here to set $c\big(D_{\pi(j)},T_{j}\big) = 1$ if $D_{\pi(j)}$ is empty, and $c\big(D_{\pi(j)},T_{j}\big) = 0$ if $T_j$ is empty, but $D_{\pi(j)}$ is not. This needs to be summed over all distinct (non-isomorphic, to be precise) permutations of the branches $D_1,D_2,\ldots,D_d$, for which $M(D)$ provides a set of representatives.
\end{itemize}
Equation~\eqref{The general recursion} will be used repeatedly in various places of this paper.

\medskip

For our next step, we have to introduce a strict analogue of even trees, which are strictly $d$-ary trees with a definition similar to even trees:

\begin{itemize}
\item The tree $H_0^d$ consists only of a single leaf.
\item If $n > 0$, we write $n - 1 = ds + b$, with $b \in \{0,1,\ldots,d-1\}$. Take $d-b$ copies of $H_s^d$ and $b$ copies of $H_{s+1}^d$, and connect a new common root to each of their roots by an edge to obtain $H_n^d$. For future purposes, let us write $s_1(n) = s_2(n) = \cdots =  s_{d-b}(n) = s$ and $s_{d-b+1}(n) = s_{d-b+2}(n) = \cdots = s_d(n) = s+1$ to encode the branches. 
\end{itemize}
See Figure~\ref{treeH33} for an example. One easily shows that $H_n^d$ is a strictly $d$-ary tree with $(d-1)n+1$ leaves for every $n$. As in the sequence of even $d$-ary trees, the leaves are as evenly distributed among the branches as possible. We first prove that $\gamma(D,H_n^d)$ converges to a positive limit for every fixed $d$-ary tree $D$, which immediately provides a lower bound on the inducibility.

\begin{figure}[htbp]\centering
\begin{tikzpicture}[thick,level distance=15mm, scale=0.75]
\tikzstyle{level 1}=[sibling distance=20mm]   
\tikzstyle{level 2}=[sibling distance=6mm]   
\node [circle,draw]{}
child {child {[fill] circle (2pt)}child {[fill] circle (2pt)}child {[fill] circle (2pt)}}
child {child {[fill] circle (2pt)}child {[fill] circle (2pt)}child {[fill] circle (2pt)}}
child {[fill] circle (2pt)};
\end{tikzpicture}
\caption{The tree $H^3_3$.}\label{treeH33}
\end{figure}
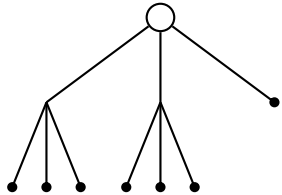

\begin{theorem}\label{Completeproportion}
For every fixed $d$-ary tree $D$, the limit
$$\eta_d(D) = \lim_{n \to \infty} \gamma(D,H_n^d)$$
exists. It can be determined recursively as follows: it is equal to $1$ when $D$ is empty or only consists of a single vertex. Otherwise, let $D_1,D_2,\ldots,D_d$ be the branches of $D$ (some possibly empty), and define $M(D)$ as in~\eqref{The general recursion}. We have
$$\eta_d(D) =  \binom{\|D\|}{\|D_1\|,\|D_2\|,\ldots,\|D_d\|} \frac{|M(D)|}{d^{\|D\|}-d} \prod_{i=1}^d \eta_d(D_i)\,.$$
Moreover, we have
$$\gamma(D,H_n^d) = \eta_d(D) + O\big(n^{-1}\big)\,,$$
where the constant implied by the $O$-term only depends on $d$ and $D$.
\end{theorem}

\begin{proof}
Define $\phi(D) = \frac{(d-1)^{\|D\|}\eta_d(D)}{\|D\|!}$, which satisfies the recursion
\begin{equation}\label{eq:phi_dec}
\phi(D) = \frac{|M(D)|}{d^{\|D\|}-d} \prod_{i=1}^d \phi(D_i)\,.
\end{equation}
Let us set $k = \|D\|$ for simplicity and use induction on $k$ to prove that there exists a positive constant $\kappa(D)$ such that
\begin{equation}\label{eq:upper_bound}
c(D,H_n^d) \leq \phi(D) n^k + \kappa(D) n^{k-1}
\end{equation}
holds for all $n \geq 0$. This is straightforward for $k = 0$, where $D$ is the empty tree, $c(D,H_n^d) = 1$ and $\phi(D) = 1$, so that we can take $\kappa(D) = 0$. Note here that we set $0^0 = 1$. Likewise, the cases $k = 1$ (where $c(D,H_n^d) = \|H_n^d\| = (d-1)n+1$ and $\phi(D) = d-1$) and $k = 2$ (where $c(D,H_n^d) = \binom{\|H_n^d\|}{2} = \binom{(d-1)n+1}{2}$ and $\phi(D) = \frac{(d-1)^2}{2}$) are easy. For the induction step, we may now assume that $k \geq 3$ and that~\eqref{eq:upper_bound} holds for all branches $D_1,D_2,\ldots,D_d$ of $D$. In other words, writing $\ell_i = \|D_i\|$, there exist constants $\kappa(D_1),\kappa(D_2),\ldots,\kappa(D_d)$ such that
$$c(D_i,H_n^d) \leq \phi(D_i) n^{\ell_i} + \kappa(D_i) n^{\ell_i-1}$$
holds for all $n \geq 0$ and all $i \in [d]$.

\medskip

Now we prove ~\eqref{eq:upper_bound} for $D$ by induction on $n$; the precise value of the constant $\kappa(D)$ will be specified later. The inequality holds trivially for $n =0$ though, regardless of the value of $\kappa(D)$. For the induction step, we use the recursion~\eqref{The general recursion}, where $T = H_n^d$. The branches $T_i$ are of the form $H_s^d$ or $H_{s+1}^d$, where $s = \lfloor \frac{n-1}{d} \rfloor$. Thus the final sum in~\eqref{The general recursion} is bounded above by
\begin{align*}
\sum_{\pi \in M(D)}~&\prod_{j=1}^d \Big( \phi(D_{\pi(j)}) \Big( \Big\lfloor \frac{n-1}{d} \Big\rfloor + 1 \Big)^{\ell_{\pi(j)}} + \kappa(D_{\pi(j)}) \Big( \Big\lfloor \frac{n-1}{d} \Big\rfloor + 1 \Big)^{\ell_{\pi(j)}-1} \Big) \\
 &= \sum_{\pi \in M(D)}~\prod_{j=1}^d \Big( \phi(D_{\pi(j)}) \Big( \frac{n}{d} \Big)^{\ell_{\pi(j)}} + O \big(n^{\ell_{\pi(j)}-1} \big) \Big) \\
&= |M(D)| \Big( \prod_{i=1}^d \phi(D_i) \Big) \frac{n^k}{d^k} + O(n^{k-1})\,,
\end{align*}
since $\ell_1 + \ell_2 + \cdots + \ell_d = k$. In view of~\eqref{eq:phi_dec}, we can write this as 
$$\phi(D) \big(1 - d^{1-k}\big) n^k + O(n^{k-1})\,.$$
Therefore, there exists a positive constant $C(D)$ such that the final sum in~\eqref{The general recursion}, applied to $T = H_n^d$, is bounded above by
$$\phi(D) \big(1 - d^{1-k}\big) n^k + C(D)n^{k-1}$$
for all $n \geq 1$. Moreover, applying the induction hypothesis (with respect to $n$) to the branches of $H_n^d$, we obtain
$$c(D,H_n^d) \leq \sum_{i=1}^{d} \Big( \phi(D) s_i(n)^k + \kappa(D) s_i(n)^{k-1} \Big) + \phi(D) \big(1 - d^{1-k}\big) n^k + C(D)n^{k-1}\,.$$
Now let us finally specify the value of $\kappa(D)$. We take it to be
$$\kappa(D) = \sup_{n \geq 1} \frac{\phi(D) \sum_{i=1}^{d} \big| s_i(n)^k - ( \frac{n}{d} )^k \big| + C(D)n^{k-1}}{n^{k-1} - \sum_{i=1}^{d} s_i(n)^{k-1}}\,.$$
To see why this constant is positive and finite, note first that the denominator is always positive, as $\sum_{i=1}^{d} s_i(n)^{k-1} \leq (\sum_{i=1}^d s_i(n) )^{k-1} = (n-1)^{k-1}$. The numerator is clearly positive, so the fraction is positive for every $n$. Moreover, since $s_i(n) = \frac{n}{d} + O(1)$ for each $i$, the numerator is $O(n^{k-1})$, and the denominator is $n^{k-1}(1-d^{2-k}) + O(n^{k-2})$. The factor $1-d^{2-k}$ is positive as we are assuming $k \geq 3$. Therefore, the quotient remains bounded as $n \to \infty$. With this definition, we obtain
\begin{align*}
c(D,H_n^d) &\leq \sum_{i=1}^{d} \Big( \phi(D) \Big( \frac{n}{d} \Big)^k + \phi(D) \Big| s_i(n)^k - \Big( \frac{n}{d} \Big)^k \Big|  + \kappa(D) s_i(n)^{k-1} \Big) \\
&\quad  + \phi(D) \big(1 - d^{1-k}\big) n^k + C(D)n^{k-1} \\
&= \phi(D) n^k + \phi(D) \sum_{i=1}^{d} \Big| s_i(n)^k - \Big( \frac{n}{d} \Big)^k \Big| + C(D)n^{k-1} + \kappa(D) \sum_{i=1}^d  s_i(n)^{k-1} \\
&\leq \phi(D) n^k +  \kappa(D) \Big( n^{k-1} - \sum_{i=1}^{d} s_i(n)^{k-1} \Big) + \kappa(D) \sum_{i=1}^d  s_i(n)^{k-1} \\
&= \phi(D) n^k + \kappa(D) n^{k-1}\,,
\end{align*}
completing the induction with respect to $n$ and thus also with respect to $k$.

\medskip

In the same fashion, one proves that there exists a constant $\lambda(D)$ such that
$$c(D,H_n^d) \geq \phi(D) n^k - \lambda(D) n^{k-1}$$
holds for all $n$, which completes the proof of our theorem. We skip the details. 
\end{proof}
Note that $\eta_d(D) = \lim_{n \to \infty} \gamma\big(D,H_n^d \big) \leq I_d(D)$ holds by definition, so Theorem~\ref{Completeproportion} provides a lower bound on the inducibility. We now show that this lower bound is in fact sharp for (among others) even trees, thereby proving Theorem~\ref{IndEvendAryTree}. This is achieved by proving a matching upper bound, which is derived in the following.

\section{Upper bounds involving branches}

We first need some notation. For a fixed $d\geq 2$ and a given $d$-ary tree $D$ with branches $D_1,D_2,\ldots,D_d$ (some of them possibly empty), we define the $d$-dimensional real function
$$Z_D(x_1,x_2,\ldots,x_d):=\frac{1}{1-\sum_{i=1}^d x_i^{\|D\|}} \sum_{\pi \in M(D)}~\prod_{j=1}^d x_{j}^{\|D_{\pi(j)}\|}\,.$$
It follows from the definition that this function is always symmetric in its variables. This is because $\pi \sim_D \pi'$ implies that
$D_{\pi(j)}$ and $D_{\pi'(j)}$ are isomorphic for all $j$, thus $\|D_{\pi(j)}\| = \|D_{\pi'(j)}\|$ for all $j$. It follows that the final product is the same for all members of an equivalence class of $\sim_D$, and we can write $Z_D$ as
\begin{equation}\label{eq:ZD_alt}
Z_D(x_1,x_2,\ldots,x_d) =\frac{1}{1-\sum_{i=1}^d x_i^{\|D\|}} \cdot \frac{|M(D)|}{d!} \sum_{\pi \in S_d}~\prod_{j=1}^d x_{j}^{\|D_{\pi(j)}\|}\,.
\end{equation}
For example, when $D$ is the even ternary tree $E_7^3$ with seven leaves, as shown in Figure~\ref{Someeventernarytrees}, the function is given by
$$Z_{E_7^3}(x_1,x_2,x_3) = \frac{x_1^3 x_2^2x_3^2+x_1^2 x_2^3x_3^2 + x_1^2 x_2^2x_3^3}{1-x_1^7-x_2^7-x_3^7}\,.$$
The following proposition bounds the inducibility of a tree $D$ in terms of the inducibilities of its branches and the function $Z_D$. 
\begin{proposition}\label{ind less prod ind}
Let $D$ be a fixed $d$-ary tree with branches $D_1,D_2,\ldots,D_d$ (some of them possibly empty). Then the following inequality holds:
\begin{equation}\label{eq:sup-ineq}
I_d(D)\leq \binom{\|D\|}{\|D_1\|,\|D_2\|,\ldots,\|D_d\|} \Big(\prod_{i=1}^d I_d(D_i) \Big) \sup_{\substack{0\leq x_1,x_2,\ldots,x_d<1\\ x_1+x_2+\cdots +x_d=1}} 	Z_D(x_1,x_2,\ldots,x_d)\,.
\end{equation}
\end{proposition}

The benefit of this proposition is that the combinatorial problem is translated to a purely analytic question. The supremum on the right side of the inequality can be determined explicitly in many cases. In order to prove the proposition, we first need a technical lemma on the supremum occurring in~\eqref{eq:sup-ineq} that will also be useful at a later point.

\begin{lemma}\label{lem:supGen}
Let $D$ be a $d$-ary tree, and let $D_1,D_2,\ldots,D_d$ be its branches (some of which might be empty). Moreover, let $m_{j}$ be the number of branches with $j$ leaves for every $j \geq 0$. We have
$$\sup_{\substack{0\leq  x_1,x_2,\ldots,x_d<1 \\ x_1+x_2+\cdots +x_d=1}}Z_D(x_1,x_2,\ldots,x_d)  \leq \frac{|M(D)| \prod_{j \geq 0} m_j!}{d!} \cdot \binom{\|D\|}{\|D_1\|,\|D_2\|,\ldots,\|D_d\|}^{-1}\,.$$
In particular, if branches with the same number of leaves are isomorphic, then we have
$$\sup_{\substack{0\leq  x_1,x_2,\ldots,x_d<1 \\ x_1+x_2+\cdots +x_d=1}}Z_D(x_1,x_2,\ldots,x_d)  \leq \binom{\|D\|}{\|D_1\|,\|D_2\|,\ldots,\|D_d\|}^{-1}\,.$$
Finally, if $D$ has only two nonempty branches $D_1,D_2$ with $\|D_1\|=1$ and $\|D_2\|>1$, then
$$\sup_{\substack{0\leq  x_1,x_2,\ldots,x_d<1 \\ x_1+x_2+\cdots +x_d=1}}Z_D(x_1,x_2,\ldots,x_d) =\frac{1}{\|D\|}\,.$$
\end{lemma}

\begin{proof}
Let us use the abbreviations $k = \|D\|$ and $\ell_i = \|D_i\|$.
Recall that we can write the function $Z_D$ as
$$Z_D(x_1,x_2,\ldots,x_d) =\frac{1}{1-\sum_{i=1}^d x_i^{k}} \cdot \frac{|M(D)|}{d!} \sum_{\pi \in S_d}~\prod_{j=1}^d x_{j}^{\ell_{\pi(j)}}\,.$$
We apply the multinomial theorem to $(x_1+x_2+\cdots+x_d)^k$ and split the resulting terms according to the exponents of $x_1,x_2,\ldots,x_d$ into permutations of $(k,0,\ldots,0)$, permutations of $(\ell_1,\ell_2,\ldots,\ell_d)$ and the rest. The terms corresponding to permutations of $(k,0,\ldots,0)$ are clearly $\sum_{i=1}^d x_i^k$. Monomials where the exponents form a permutation of $(\ell_1,\ell_2,\ldots,\ell_d)$ have a coefficient of $\binom{k}{\ell_1,\ell_2,\ldots,\ell_d}$, and the sum of all such monomials can be expressed as
$$\frac{1}{\prod_{j \geq 0} m_j!} \sum_{\pi \in S_d}\prod_{j=1}^d x_{j}^{\ell_{\pi(j)}}\,,$$
since each of them occurs precisely $\prod_{j \geq 0} m_j!$ times in the sum over all permutations in $S_d$. So the contribution to the expansion of $(x_1+x_2+\cdots+x_d)^k$ according to the multinomial theorem can be expressed as
$$\frac{1}{\prod_{j \geq 0} m_j!} \binom{k}{\ell_1,\ell_2,\ldots,\ell_d} \sum_{\pi \in S_d}\prod_{j=1}^d x_{j}^{\ell_{\pi(j)}}\,.$$
The remaining terms corresponding to exponents that are not permutations of $(k,0,\ldots,0)$ or $(\ell_1,\ell_2,\ldots,\ell_d)$
are clearly nonnegative whenever the variables $x_1,x_2,\ldots,x_d$ are, so we obtain
$$(x_1+x_2 + \cdots + x_d)^k \geq \sum_{i=1}^d x_i^k + \frac{1}{\prod_{j \geq 0} m_j!} \binom{k}{\ell_1,\ell_2,\ldots,\ell_d} \sum_{\pi \in S_d} \prod_{j=1}^d x_{j}^{\ell_{\pi(j)}}\,.$$
If additionally $x_1+x_2+\cdots+x_d = 1$, then we can easily manipulate this to get
$$Z_D(x_1,x_2,\ldots,x_d) = \frac{1}{1-\sum_{i=1}^d x_i^k} \cdot \frac{|M(D)|}{d!} \sum_{\pi \in S_d} \prod_{j=1}^d x_{j}^{\ell_{\pi(j)}} \leq \frac{|M(D)| \prod_{j \geq 0} m_j!}{d!} \binom{k}{\ell_1,\ell_2,\ldots,\ell_d}^{-1}\,,$$
which proves the first part.

\medskip

If we assume that two branches of $D$ are isomorphic if and only if they have the same number of leaves, then the equivalence relation $\sim_D$ is given by
$$\pi \sim_D \pi'\ \Longleftrightarrow \ell_{\pi(j)} = \ell_{\pi'(j)} \text{ for all $j$.}$$
Accordingly, the elements of $M(D)$ correspond to all distinct permutations of $(\ell_1,\ell_2,\ldots,\ell_d)$, and we have
$$|M(D)| = \frac{d!}{\prod_{j \geq 0} m_j!}\,,$$
giving us the second part of the lemma.

\medskip

For the proof of the final part of the lemma, we merely need to show that the upper bound can be reached in the limit for a suitable sequence of vectors $(x_1,x_2,\ldots,x_d)$.  Note that under the given conditions, the function $Z_D$ is given by
$$Z_D(x_1,x_2,\ldots,x_d) =\frac{\sum_{\{i,j\}\subseteq [d]} \big(x_{i}x_{j}^{k-1} +x_{i}^{k-1} x_{j}\big)}{1-\sum_{i=1}^d x_i^k} =\frac{\big(\sum_{i=1}^d x_i \big) \big(\sum_{i=1}^d x_i^{k-1} \big) - \big(\sum_{i=1}^d x_i^k \big)}{1-\sum_{i=1}^d x_i^k}\,.$$
If we set $x_1 = x_2 = \cdots = x_{d-1} = \epsilon$ and $x_d = 1 - (d-1)\epsilon$, the numerator is readily seen to be $(d-1)\epsilon + O(\epsilon^2)$, while the denominator is $(d-1)k\epsilon + O(\epsilon^2)$. Hence we have
\begin{align*}
\lim_{\epsilon \to  0^+} Z_D(\epsilon,\epsilon,\ldots,\epsilon,1-(d-1)\epsilon) = \frac{1}{k}\,,
\end{align*}
which completes the proof of the lemma.
\end{proof}

\begin{proof}[Proof of Proposition~\ref{ind less prod ind}]
If $D$ has only two leaves, then we can assume that both $D_1$ and $D_2$ are single vertices, and all other branches empty. We have $I_d(D) = I_d(D_1) = I_d(D_2) = \cdots = I_d(D_d) = 1$. Moreover,
$$Z_D(x_1,x_2,\ldots,x_d) = \frac{\sum_{\{i,j\}\subseteq [d]} x_{i}x_{j}}{1-\sum_{i=1}^d x_i^2} = \frac{(\sum_{i=1}^d x_i)^2 - \sum_{i=1}^d x_i^2}{2(1-\sum_{i=1}^d x_i^2)}\,.$$
If $x_1+x_2+\cdots+x_d = 1$, this actually simplifies to $\frac12$, so the supremum in the inequality is $\frac12$, and we have equality.
So we can assume that $D$ has more than two leaves. As before, let us use the abbreviations $k = \|D\|$ and $\ell_i = \|D_i\|$. 
We know from the proof of Theorem~3 in~\cite{AudaceStephanPaper1} that
\begin{equation}\label{improveTo}
0\leq \max_{\substack{\|T\|=n \\T~\text{$d$-ary tree}}}\gamma(D,T) - I_d(D)\leq \frac{k(k-1)}{n}\,
\end{equation} 
for all $n \geq k$. Consider a sequence $T_1,T_2,\ldots$ of $d$-ary trees such that $\|T_n\| \to \infty$ as $n \to \infty$ and $c(D,T_n)$ is the maximum of $c(D,T)$ over all trees $T$ with the same number of leaves as $T_n$.
Denote the branches of $T_n$ by $T_{n,1},T_{n,2},\ldots,T_{n,d}$ (some of these branches are allowed to be empty). One can assume that $T_{n,1}$ is the branch of $T_n$ with the greatest number of leaves for every $n$. Set $\alpha_{n,i}:=\|T_{n,i}\|/\|T_n\|$ for every $i\in [d]$ and every $n$ (the proportion of leaves belonging to $T_{n,i}$), and set $\beta_n = 1 - \alpha_{n,1}$. We distinguish two cases based on whether $\beta_n$ is ``small'' or ``large'' in the limit.

\medskip

\textbf{Case 1}: Suppose that $\beta_n$ is bounded below by a positive constant $\delta$ as $n\to \infty$. We can assume that $\delta \leq \frac{1}{d}$. Note that $\beta_n$ is automatically bounded above by $\frac{d-1}{d}$ by definition. It follows that
\begin{equation}\label{eq:denom_estimate}
1- \sum_{i=1}^d \alpha_{n,i}^k \geq 1- \alpha_{n,1}^k - \Big(\sum_{i=2}^d \alpha_{n,i}\Big)^k 
= 1- (1-\beta_n)^k-\beta_n^k \geq 1- (1-\delta)^k-\delta^k
\end{equation}
for all $n$, since the function $x \mapsto 1-(1-x)^k-x^k$ is increasing for $x \in [0,\frac12]$ and decreasing for $x \in [\frac12,1]$. Now we apply recursion~\eqref{The general recursion}:
$$c(D,T_n)=\sum_{i=1}^d c(D,T_{n,i}) + \sum_{\pi \in M(D)}~\prod_{j=1}^d c(D_{\pi(j)},T_{n,j})\,.$$
In view of~\eqref{improveTo}, it gives us
\begin{align*}
I_d(D)  \binom{\|T_n\|}{k} \leq c(D,T_n) &\leq \sum_{i=1}^d \Big(I_d(D) + \frac{k(k-1)}{\|T_{n,i}\|} \Big)\binom{\|T_{n,i}\|}{k} \\
&\quad + \sum_{\pi \in M(D)}~\prod_{j=1}^d \Big(I_d(D_{\pi(j)}) + \frac{\ell_{\pi(j)}(\ell_{\pi(j)}-1)}{\|T_{n,j}\|} \Big)\binom{\|T_{n,j}\|}{\ell_{\pi(j)}}\,, 
\end{align*}
which implies that
\begin{align*}
I_d(D)  \binom{\|T_n\|}{k} &\leq \sum_{i=1}^d \Big( I_d(D) \frac{\|T_{n,i}\|^k}{k!} + \frac{\|T_{n,i}\|^{k-1}}{(k-2)!}\Big)\\
&\quad + \sum_{\pi \in M(D)}~\prod_{j=1}^d \Big( I_d(D_{\pi(j)}) \frac{\|T_{n,j}\|^{\ell_{\pi(j)}}}{\ell_{\pi(j)}!} + N(T_{n,j},D_{\pi(j)}) \Big)\,,
\end{align*}
where $N(T_{n,j},D_{\pi(j)})$ is equal to $\|T_{n,j}\|^{\ell_{\pi(j)}-1}/(\ell_{\pi(j)}-2)!$ if $\ell_{\pi(j)} \geq 2$, and $0$ otherwise. Consequently,
$$\Big(\|T_n\|^k  - \sum_{i=1}^d \|T_{n,i}\|^k \Big)I_d(D) \leq k! \sum_{\pi \in M(D)}~\prod_{j=1}^d I_d(D_{\pi(j)}) \frac{\|T_{n,j}\|^{\ell_{\pi(j)}}}{\ell_{\pi(j)}!} + \mathcal{O}(\|T_n\|^{k-1})$$
as $\|T_{n,j}\| < \|T_n\|$ for all $j \in [d]$ and all $n$. Dividing through by $\|T_n\|^k$, we get
$$	
\Big(1 - \sum_{i=1}^d \alpha_{n,i}^k \Big)I_d(D) \leq \frac{k!}{\ell_1! \ell_2! \cdots \ell_d!} \sum_{\pi \in M(D)}~\prod_{j=1}^d I_d(D_{\pi(j)}) \alpha_{n,j}^{\ell_{\pi(j)}}+ \mathcal{O}(\|T_n\|^{-1})\,.$$

Now using the fact that $1- \sum_{i=1}^d \alpha_{n,i}^k$ is bounded below by a positive constant as $n\to \infty$ by~\eqref{eq:denom_estimate}, we deduce that
\begin{align*}
I_d(D) &\leq \Big(\prod_{i=1}^d I_d(D_i)\Big) \binom{k}{\ell_1,\ell_2,\ldots,\ell_d} Z_D(\alpha_{n,1},\alpha_{n,2},\ldots,\alpha_{n,d}) + \mathcal{O}(\|T_n\|^{-1}) \\
&\leq \Big(\prod_{i=1}^d I_d(D_i)\Big) \binom{k}{\ell_1,\ell_2,\ldots,\ell_d} \sup_{\substack{0\leq x_1,x_2,\ldots,x_d<1 \\ x_1+x_2+\cdots + x_d=1}} Z_D(x_1,x_2,\ldots,x_d) + \mathcal{O}(\|T_n\|^{-1})\,.
\end{align*}
Finally, we take the limit as $n\to \infty$, giving us the desired result.

\medskip

\textbf{Case 2}: If $\beta_n$ is not bounded below by a positive constant, then we can assume (without loss of generality, by considering a subsequence if necessary)  that the limit of $\beta_n$ is actually $0$ as $n\to \infty$. Denote by $T_n \backslash T_{n,1}$ the tree that is obtained by removing the branch $T_{n,1}$ from $T_n$ (and possibly the root of $T_n$ if there is only one other nonempty branch).

\medskip

\emph{Claim 1:} We claim that the number of copies of $D$ in $T_n$ that involve more than one leaf of $T_n \backslash T_{n,1}$ is at most of order $\mathcal{O}(\beta_n^2 \|T_n\|^k)$.

\medskip

For the proof of the claim, note that by definition, the number of copies of $D$ in $T_n$ that involve more than one leaf of $T_n \backslash T_{n,1}$ is at most
\begin{align*}
\sum_{j=2}^k \binom{\|T_n\|-\|T_{n,1}\|}{j} \binom{\|T_{n,1}\|}{k-j} &\leq \sum_{j=2}^k  \frac{(\|T_n\|-\|T_{n,1}\|)^j \|T_{n,1}\|^{k-j} }{j! (k-j)!} \\
& = \|T_n\|^k (1-\alpha_{n,1})^2 \sum_{j=2}^k \frac{(1-\alpha_{n,1})^{j-2} \alpha_{n,1}^{k-j}}{j!(k-j)!}\\
&  \leq \|T_n\|^k  \beta_n^2 \sum_{j=2}^k \frac{1}{j!(k-j)!}\,. 
\end{align*}
This completes the proof of the claim. It follows that the proportion of copies of $D$ in $T_n$ that involve more than one leaf of $T_n \backslash T_{n,1}$ is of order at most $\mathcal{O}(\beta_n^2)$ among all subsets of $k$ leaves of $T_n$.

\medskip

\emph{Claim 2:} We further claim that $D$ must have only two nonempty branches, one of which is a single leaf.

\medskip

Indeed, suppose that $D$ does not have this shape. Then the subsets of leaves of $T_n$ that induce a copy of $D$ come in two types: either the $k$ leaves are all leaves of $T_{n,1}$, or more than one of the $k$ leaves is a leaf of $T_n \backslash T_{n,1}$. So this gives us
\begin{equation}\label{NotShape}
c(D,T_n) = c(D,T_{n,1}) + \mathcal{O}(\beta_n^2 \|T_n\|^k)
\end{equation}
by Claim~1. It was established in the proof of~\eqref{improveTo} (see \cite[Theorem 3]{AudaceStephanPaper1}) that 
$$0\leq \max_{\substack{\|T^{\prime}\|=j \\T^{\prime}~\text{$d$-ary tree}}}\gamma(D,T^{\prime}) -\max_{\substack{\|T^{\prime \prime}\|=j+1 \\T^{\prime \prime}~\text{$d$-ary tree}}}\gamma(D,T^{\prime \prime})\leq \frac{k(k-1)}{j(j+1)}\,.$$
Summing all these inequalities for $j=m,m+1,\ldots,n-1$, we find that
$$0\leq \max_{\substack{\|T^{\prime}\|=m \\T^{\prime}~\text{$d$-ary tree}}}\gamma(D,T^{\prime}) -\max_{\substack{\|T^{\prime \prime}\|=n \\T^{\prime \prime}~\text{$d$-ary tree}}}\gamma(D,T^{\prime \prime}) \leq k(k-1)\Big(\frac{1}{m}-\frac{1}{n}\Big)\,.$$
Thus we have
$$\max_{\substack{\|T^{\prime}\|=m \\T^{\prime}~\text{$d$-ary tree}}}\gamma(D,T^{\prime}) -\max_{\substack{\|T^{\prime \prime}\|=n \\T^{\prime \prime}~\text{$d$-ary tree}}}\gamma(D,T^{\prime \prime})=\mathcal{O}\Big(\frac{n-m}{mn}\Big)$$
as $m\leq n$ and $m \to \infty$. In particular, since $T_n$ was assumed to contain the maximum number of copies of $D$ among all trees of the same size,
\begin{equation}\label{eq:change}
\gamma(D,T_{n,1}) - \gamma(D,T_n) \leq \max_{\substack{\|T^{\prime}\|=\|T_{n,1}\| \\T^{\prime}~\text{$d$-ary tree}}}\gamma(D,T^{\prime}) - \gamma(D,T_n) = \mathcal{O}\Big(\frac{\|T_n\|-\|T_{n,1}\|}{\|T_n\|\cdot \|T_{n,1}\|}\Big)\,.
\end{equation}
Using \eqref{eq:change}, formula~\eqref{NotShape} implies that
$$c(D,T_n) \leq \frac{\binom{\|T_{n,1}\|}{k}}{\binom{\|T_n\|}{k}}c(D,T_n) +  \mathcal{O}\Big(\|T_{n,1}\|^k \cdot \frac{\|T_n\|-\|T_{n,1}\|}{\|T_n\|\cdot \|T_{n,1}\|} + \beta_n^2 \|T_n\|^k \Big)\,.$$
Thus
$$\Big(1-\frac{\binom{\|T_{n,1}\|}{k}}{\binom{\|T_n\|}{k}} \Big)c(D,T_n) \leq  \mathcal{O}\big(\beta_n \|T_{n,1}\|^{k-1}  + \beta_n^2 \|T_n\|^k\big)\,,$$
and using the asymptotic formula
\begin{equation}\label{asymForm}
\binom{\|T_n\|}{k} - \binom{\|T_{n,1}\|}{k} \sim (\|T_n\|-\|T_{n,1}\|) \frac{\|T_n\|^{k-1}}{(k-1)!} = \frac{\|T_n\|^k \beta_n}{(k-1)!}\,,
\end{equation}
which holds since $\|T_n\| \sim \|T_{n,1}\|$, we derive that
\begin{align*}
\gamma(D,T_n) \leq  \mathcal{O}(\|T_n\|^{-1} + \beta_n)\,.
\end{align*}
Therefore
\begin{align*}
I_d(D)=\lim_{n \to \infty} \gamma(D,T_n) \leq 0
\end{align*}
as $\lim_{n \to \infty} \beta_n =0$. This contradicts the fact that $I_d(D)$ is strictly positive (which was mentioned in the introduction and also follows from Theorem~\ref{Completeproportion}). Thus the proof of our second claim is complete.

\medskip

Now we can assume that $D$ has only two nonempty branches, one of which ($D_1$, say) is the tree that has only one vertex. Since we are assuming that $k>2$, the second nonempty branch $D_2$ of $D$ has at least two leaves. Using Claim~1, we get
$$c(D,T_n) = c(D,T_{n,1}) + (\|T_n\|-\|T_{n,1}\|)c(D_2,T_{n,1}) + \mathcal{O}(\beta_n^2 \|T_n\|^k)\,.$$
Following the same course of reasoning used to prove Claim~2, it is not difficult to see that
$$\Big(1-\frac{\binom{\|T_{n,1}\|}{k}}{\binom{\|T_n\|}{k}} \Big)c(D,T_n) \leq  (\|T_n\|-\|T_{n,1}\|)\binom{\|T_{n,1}\|}{k-1} \gamma(D_2,T_{n,1})+ \mathcal{O}\big( \beta_n \|T_n\|^{k-1} + \beta_n^2 \|T_n\|^k \big)\,.$$
It follows from the asymptotic formula~\eqref{asymForm} now that
$$\gamma(D,T_n) - \gamma(D_2,T_{n,1}) \leq  \mathcal{O}(\|T_n\|^{-1} + \beta_n)\,.$$
Applying $\liminf$ to both sides of this inequality, we get
$$I_d(D) - \limsup_{n\to \infty} \gamma(D_2,T_{n,1}) = \liminf_{n \to \infty} \big( \gamma(D,T_n) -  \gamma(D_2,T_{n,1})\big) \leq 0\,,$$
which implies that
$$I_d(D) \leq \limsup_{n\to \infty} \gamma(D_2,T_{n,1}) \leq I_d(D_2)\,.$$
This completes the proof of the proposition once we invoke the final part of Lemma~\ref{lem:supGen}.
\end{proof}
The following corollary is a direct consequence of Proposition~\ref{ind less prod ind} combined with the first and second part of Lemma~\ref{lem:supGen}:
\begin{corollary}\label{indA}
Let $D$ be a $d$-ary tree, and let $D_1,D_2,\ldots,D_d$ be its branches (some of which might be empty). Moreover, let $m_{j}$ be the number of branches with $j$ leaves for every $j \geq 0$. Then we have
$$I_d(D)\leq \frac{|M(D)| \prod_{j \geq 0} m_j!}{d!} \prod_{i=1}^d I_d(D_i)\,.$$
If branches with the same number of leaves are isomorphic, then this reduces to
$$I_d(D)\leq \prod_{i=1}^d I_d(D_i)\,.$$
\end{corollary}

Further improvements rely on our ability to determine (or estimate) the supremum over the function $Z_D$ occurring in Proposition~\ref{ind less prod ind}. This will be achieved for a special class of trees in the following section. 

\section{Balanced trees}\label{Sec:balance}

Recall that $\eta_d(D)$, as defined in Theorem~\ref{Completeproportion}, provides a lower bound on the inducibility: $\eta_d(D) \leq I_d(D)$. Simple instances where equality holds are the empty tree or trees with only one or two leaves. It turns out that there are many more such cases, which is a consequence of the following theorem:

\begin{theorem}\label{thm:supcond}
Let $D$ be a $d$-ary tree with branches $D_1,D_2,\ldots,D_d$ (some of which may be empty). If $ I_d(D_i) = \eta_d(D_i)$ for all branches and the supremum of $Z_D(x_1,x_2,\ldots,x_d)$ under the conditions $0 \leq x_i < 1$ and $x_1+x_2+\cdots+x_d = 1$ is attained when $x_1 = x_2 = \cdots = x_d = \frac{1}{d}$, i.e.,
$$\sup_{\substack{0\leq x_1,x_2,\ldots,x_d<1\\ x_1+x_2+\cdots +x_d=1}} 	Z_D(x_1,x_2,\ldots,x_d) = Z_D \Big(\frac1{d},\frac1{d},\ldots,\frac1{d}\Big) = \frac{|M(D)|}{d^{\|D\|}-d}\,,$$
then we also have $I_d(D) = \eta_d(D)$.
\end{theorem}

\begin{proof}
Let us compare the recursion for $\eta_d$,
\begin{equation}\label{eq:supcond1}
\eta_d(D) =  \binom{\|D\|}{\|D_1\|,\|D_2\|,\ldots,\|D_d\|} \frac{|M(D)|}{d^{\|D\|}-d} \prod_{i=1}^d \eta_d(D_i)\,,
\end{equation}
to the upper bound in Proposition~\ref{ind less prod ind}:
\begin{equation}\label{eq:supcond2}
I_d(D)\leq \binom{\|D\|}{\|D_1\|,\|D_2\|,\ldots,\|D_d\|} \Big(\prod_{i=1}^d I_d(D_i) \Big) \sup_{\substack{0\leq x_1,x_2,\ldots,x_d<1\\ x_1+x_2+\cdots +x_d=1}} 	Z_D(x_1,x_2,\ldots,x_d)\,.
\end{equation}
The similarities are obvious. Plugging $x_1 = x_2 = \cdots = x_d = \frac{1}{d}$ into the representation~\eqref{eq:ZD_alt}, we obtain
$$Z_D\Big(\frac1{d},\frac1{d},\ldots,\frac1{d}\Big) = \frac{1}{1-\sum_{i=1}^d d^{-\|D\|}} \cdot \frac{|M(D)|}{d!} \sum_{\pi \in S_d} d^{-\|D\|} = \frac{|M(D)|}{d^{\|D\|}-d}\,.$$
Thus we can combine~\eqref{eq:supcond1} and~\eqref{eq:supcond2}, giving us
\begin{align*}
\eta_d(D) \leq I_d(D) &\leq \binom{\|D\|}{\|D_1\|,\|D_2\|,\ldots,\|D_d\|} \frac{|M(D)|}{d^{\|D\|}-d} \prod_{i=1}^d I_d(D_i) \\
&= \binom{\|D\|}{\|D_1\|,\|D_2\|,\ldots,\|D_d\|} \frac{|M(D)|}{d^{\|D\|}-d} \prod_{i=1}^d \eta_d(D_i) = \eta_d(D)\,,
\end{align*}
which implies that $I_d(D) = \eta_d(D)$.
\end{proof}

Let us now define a class of trees for which the supremum condition of Theorem~\ref{thm:supcond} is satisfied. A \emph{balanced} $d$-ary tree is a $d$-ary tree whose branches $D_1,D_2,\ldots,D_d$ (some of which may be empty) satisfy $\big| \| D_i \| - \|D_j\| \big| \leq 1$ for all $i,j$, i.e., the number of leaves in two different branches differs at most by one. In particular, this means that a balanced $d$-ary tree is either a star or has root degree $d$. 
Figure~\ref{1newbounevenf} shows an example of a balanced $4$-ary tree.
\begin{figure}[htbp]\centering 
\begin{tikzpicture}[thick,level distance=11mm, scale=0.85, transform shape]
\tikzstyle{level 1}=[sibling distance=30mm]   
\tikzstyle{level 2}=[sibling distance=10mm]
\tikzstyle{level 3}=[sibling distance=5mm]
\node [circle,draw]{}
child {child {child {[fill] circle (2pt)}child {[fill] circle (2pt)}child {[fill] circle (2pt)}child {[fill] circle (2pt)}}child {[fill] circle (2pt)}}
child {child {[fill] circle (2pt)}child {[fill] circle (2pt)}child {child {[fill] circle (2pt)}child {[fill] circle (2pt)}}child {[fill] circle (2pt)}}
child {child {child {[fill] circle (2pt)}child {[fill] circle (2pt)}}child {child {[fill] circle (2pt)}child {[fill] circle (2pt)}}}
child {child {[fill] circle (2pt)}child {[fill] circle (2pt)}child {child {[fill] circle (2pt)}child {[fill] circle (2pt)}}};
\end{tikzpicture}
\caption{A balanced $4$-ary tree.}\label{1newbounevenf}
\end{figure}
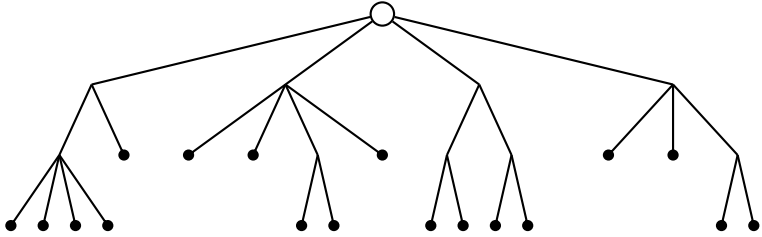

It turns out that balanced $d$-ary trees always satisfy the supremum condition of Theorem~\ref{thm:supcond}. Among other things, this will imply Theorem~\ref{IndEvendAryTree}.

\begin{lemma}\label{EvenDTlema}
For every balanced $d$-ary tree $D$ with branches $D_1,D_2,\ldots,D_d$ (some of which may be empty), we have
$$\sup_{\substack{0\leq x_1,x_2,\ldots,x_d<1\\ x_1+x_2+\cdots +x_d=1}}
	Z_{D}(x_1,x_2,\ldots,x_d) = \frac{|M(D)|}{d^{\|D\|}-d}\,.$$
\end{lemma}

In the proof of this lemma, we rely on \emph{Muirhead's inequality} (see~\cite[p.~44-45]{hardy1952inequalities}). Let $A = (a_1,a_2,\ldots, a_d)$ and $B = (b_1, b_2, \ldots, b_d)$ be vectors of real numbers with $a_1\geq a_2\geq \cdots \geq a_d$ and $b_1\geq b_2\geq \cdots \geq b_d$. We say that the vector $A$ \emph{majorizes} the vector $B$ if $\sum_{i=1}^d a_i= \sum_{i=1}^d b_i$ and for every $j\in \{1,2,\ldots,d-1\}$, 
$$\sum_{i=1}^j a_i \geq  \sum_{i=1}^j b_i\,.$$
Muirhead's inequality states that for nonnegative real numbers $x_1,x_2,\ldots,x_d$, we have
$$\sum_{\pi \in S_d} \prod_{i=1}^d x_{i}^{a_{\pi(i)}} \geq \sum_{\pi \in S_d} \prod_{i=1}^d x_{i}^{b_{\pi(i)}}$$
if $A$ majorizes $B$, the sum being over all permutations of $[d] = \{1,2,\ldots,d\}$. For our purposes, the following special case is particularly relevant: let $k$ be a positive integer, and write it as $k = ds+b$, with $b \in \{0,1,\ldots,d-1\}$. It is not difficult to see that the vector $(s+1,\ldots,s+1,s,\ldots,s)$ ($b$ copies of $s+1$, followed by $d-b$ copies of $s$) is majorized by all other vectors of $d$ nonnegative integers with sum $k$.
Let us now get to the proof of Lemma~\ref{EvenDTlema}.

\begin{proof}[Proof of Lemma~\ref{EvenDTlema}]
Let us write $k = \|D\|$ as in previous proofs. Since $D$ is balanced, there exists a positive integer $s$ such that each branch of $D$ contains either $s$ or $s+1$ leaves. Writing $k = d s+b$, where $b \in \{0,1,\ldots,d-1\}$, we have $b$ branches with $s+1$ leaves, and $d-b$ branches with $s$ leaves.

We write the function $Z_D$ according to~\eqref{eq:ZD_alt} as
$$Z_D(x_1,x_2,\ldots,x_d) =\frac{1}{1-\sum_{i=1}^d x_i^k} \cdot \frac{|M(D)|}{d!} \sum_{\pi \in S_d}~\prod_{j=1}^d x_{j}^{\|D_{\pi(j)}\|}\,.$$
As mentioned before, the vector of branch sizes $(\|D_1\|,\|D_2\|,\ldots,\|D_d\|) = (s+1,\ldots,s+1,s,\ldots,s)$ (without loss of generality in decreasing order) is majorized by all other ordered nonnegative integer vectors of the same length and sum. We expand
$$\Big( \sum_{i=1}^d x_i \Big)^k -\sum_{i=1}^d x_i^k$$
by means of the multinomial theorem and group the terms according to the vector of exponents. Each of the resulting groups has the form
$$C \sum_{\pi \in S_d} ~\prod_{j=1}^d x_{j}^{a_{\pi(j)}}$$
for a suitable constant $C$ and a vector $(a_1,a_2,\ldots,a_d)$ whose sum of entries is $k$. Since each vector $(a_1,a_2,\ldots,a_d)$ majorizes $(\|D_1\|,\|D_2\|,\ldots,\|D_d\|)$, we can apply Muirhead's theorem repeatedly to obtain
$$\Big( \sum_{i=1}^d x_i \Big)^k -\sum_{i=1}^d x_i^k \geq \frac{d^k - d}{d!} \sum_{\pi \in S_d}~\prod_{j=1}^d x_{j}^{\|D_{\pi(j)}\|}\,.$$
If the sum of the $x_i$s is equal to $1$, then this immediately yields
$$Z_D(x_1,x_2,\ldots,x_d) =\frac{1}{1-\sum_{i=1}^d x_i^k} \cdot \frac{|M(D)|}{d!} \sum_{\pi \in S_d}~\prod_{j=1}^d x_{j}^{\|D_{\pi(j)}\|} \leq \frac{|M(D)|}{d^k-d}\,.$$
Equality holds in Muirhead's inequality when all the $x_i$s are equal, so this is also the case for our inequality: when $x_1 = x_2 = \cdots = x_d = \frac{1}{d}$, the upper bound is attained (as we have already seen in the proof of Theorem~\ref{thm:supcond}). This proves the lemma.
\end{proof}

The following theorem is now straightforward.

\begin{theorem}\label{upperonBalnce}
For a balanced $d$-ary tree $D$ with branches $D_1,D_2,\ldots,D_d$ (some of which may be empty), the inequality
$$I_d(D) \leq \frac{|M(D)|}{d^{\|D\|}-d} \binom{\|D\|}{\|D_1\|,\|D_2\|,\ldots,\|D_d\|} \prod_{i=1}^d I_d(D_i)$$ 
holds for every $d$.
Furthermore, if $I_d(D_i) = \eta_d(D_i)$ for all $i$, then we also have $I_d(D) = \eta_d(D)$.
\end{theorem}

\begin{proof}
The first part is a consequence of Proposition~\ref{ind less prod ind} and Lemma~\ref{EvenDTlema}. The second part
follows from Theorem~\ref{thm:supcond} together with Lemma~\ref{EvenDTlema}.
\end{proof}

The upper bound in Theorem~\ref{upperonBalnce} can be extended to trees which have fewer than $d$ nonempty branches, but are otherwise balanced (i.e., the number of leaves in any two nonempty branches differs at most by $1$). In this case, the same approach yields the inequality
$$I_d(D) \leq \frac{|M(D)|}{\Sigma(D)} \binom{\|D\|}{\|D_1\|,\|D_2\|,\ldots,\|D_d\|} \prod_{i=1}^d I_d(D_i)\,,$$
where $\Sigma(D)$ is defined as follows: let $r$ be the number of nonempty branches of $D$, and let $V(D)$ be the set of all vectors $(k_1,k_2,\ldots,k_d)$ with $0 \leq k_i < \|D\|$ for all $i$, $k_1+k_2+\cdots+k_d = \|D\|$, and at least $d-r$ of the entries $k_i$ equal to $0$. Then
$$\Sigma(D) = \sum_{(k_1,k_2,\ldots,k_d) \in V(D)} \binom{\|D\|}{k_1,k_2,\ldots,k_d}\,.$$
Let us now look at an application of Theorem~\ref{upperonBalnce}. Since even trees are balanced, Theorem~\ref{IndEvendAryTree} follows immediately.

\begin{proof}[Proof of Theorem~\ref{IndEvendAryTree}]
Let us define $E_0^d$ to be the empty tree, which is consistent with the recursive definition. Setting $c_k = I_d(E_k^d)/k!$, we have $c_0 = c_1 = 1$. By Theorem~\ref{upperonBalnce} and the definition of even trees, we have
$$c_{ds+b} = \frac{I_d(E_{ds+b}^d)}{(ds+b)!} = \frac{|M(E_{ds+b}^d)|}{d^{ds+b}-d} \Big( \frac{I_d(E_s^d)}{s!} \Big)^{d-b} \Big( \frac{I_d(E_{s+1}^d)}{(s+1)!} \Big)^{b}\,.$$
The even tree $E_{ds+b}^d$ has two types of branches ($E_s^d$ and $E_{s+1}^d$), occurring $d-b$ and $b$ times respectively. Thus the elements of $M(E_{ds+b}^d)$ correspond precisely to the $b$-element subsets of~$[d]$. So $|M(E_{ds+b}^d)| = \binom{d}{b}$, and it follows that
$$c_{ds+b} = \frac{\binom{d}{b}}{d^{ds+b}-d} c_s^{d-b}c_{s+1}^b\,,$$
which is precisely the recursion stated in Theorem~\ref{IndEvendAryTree}.
\end{proof}

Looking at small instances, we find evidence that the even $d$-ary tree $E^d_n$ always has the greatest number of copies of the tree $E^d_k$ over all $n$-leaf $d$-ary trees:
\begin{conjecture}
Let $d\geq 2$ and $k\geq 1$ be two fixed positive integers. Then we have
$$\max_{\substack{\|T\|=n \\T~\text{$d$-ary tree}}}c(E^d_k,T)=c(E^d_k,E^d_n)$$
for every $n \geq 1$.
\end{conjecture}

A complete $d$-ary tree is a strictly $d$-ary tree in which all leaves are at the same distance from the root; the complete $d$-ary tree with $d^h$ leaves (whose distance from the root is $h$) is denoted by $C_h^d$. It is not difficult to see that complete $d$-ary trees are precisely the even trees of the form $E_{d^h}^d$, see for instance $E_9^3$ in Figure~\ref{Someeventernarytrees}. As a corollary of Theorem~\ref{IndEvendAryTree}, we obtain an explicit formula for the inducibility of a complete $d$-ary tree.
\begin{corollary}\label{ind of CDh in d ary trees}
For the complete $d$-ary tree of height $h$, we have
$$I_d\big(C^d_h\big)=(d^h)! \prod_{i=0}^{h-1}\Big(d^{d^{h-i}}-d \Big)^{-d^i}\,.$$
\end{corollary}

\begin{proof}
Note that $I_d(C^d_h) = (d^h)! c_{d^h}$, and the recursion
$$c_{d^h} = \frac{1}{d^{d^h}-d} c_{d^{h-1}}^d$$
holds. The stated formula follows easily by induction.
\end{proof}

Balanced trees are by far not the only trees that satisfy the supremum condition of Theorem~\ref{thm:supcond}. For example, one can show that all binary trees where one branch has $\ell \geq 2$ leaves and the other $\ell+2$ leaves satisfy it. This implies, among other instances, that the binary tree $T_1$ in Figure~\ref{fig:another_example} has inducibility $I_2(T_1) = \eta_2(T_1) = \frac{45}{217}$. The tree $T_2$ in the same figure can also be shown to satisfy the conditions of Theorem~\ref{thm:supcond} (even though it is not balanced), and one obtains $I_3(T_2) = \eta_3(T_2) = \frac{15}{121}$.

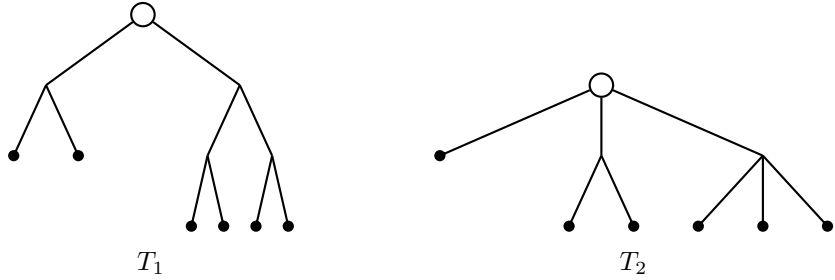
\begin{figure}[htbp]\centering 
\begin{subfigure}[b]{0.4\textwidth} \centering 
\begin{tikzpicture}[thick,level distance=11mm, scale=0.85, transform shape]
\tikzstyle{level 1}=[sibling distance=30mm]   
\tikzstyle{level 2}=[sibling distance=10mm]
\tikzstyle{level 3}=[sibling distance=5mm]
\node [circle,draw]{}
child {child {[fill] circle (2pt)}child {[fill] circle (2pt)}}
child {child {child {[fill] circle (2pt)}child {[fill] circle (2pt)}}child {child {[fill] circle (2pt)}child {[fill] circle (2pt)}}};
\end{tikzpicture}\quad
\caption{$T_1$}
\end{subfigure}
\begin{subfigure}[b]{0.4\textwidth} \centering 
\begin{tikzpicture}[thick,level distance=11mm, scale=0.85, transform shape]
\tikzstyle{level 1}=[sibling distance=25mm]   
\tikzstyle{level 2}=[sibling distance=10mm]
\node [circle,draw]{}
child {[fill] circle (2pt)}
child {child {[fill] circle (2pt)}child {[fill] circle (2pt)}}
child {child {[fill] circle (2pt)}child {[fill] circle (2pt)}child {[fill] circle (2pt)}};
\end{tikzpicture}
\caption{$T_2$}
\end{subfigure}
\caption{Further examples satisfying the conditions of Theorem~\ref{thm:supcond}.}\label{fig:another_example}
\end{figure}

Unfortunately, it does not seem easy to characterise the cases when the supremum condition of Theorem~\ref{thm:supcond} is satisfied. An answer to this open question would be extremely useful.

\medskip

We conclude this section with a result on the speed of convergence of the maximum density over strictly $d$-ary trees to the inducibility.

\begin{theorem}\label{thm:speed}
If $I_d(D) = \eta_d(D)$, then we have
\begin{equation}\label{eq:strict_speed}
\max_{\substack{\|T\| = (d-1)n+1 \\ T \text{ strictly } d\text{-ary tree}}} \gamma(D,T) = I_d(D) + \mathcal{O}(n^{-1}).
\end{equation}
\end{theorem}

\begin{proof}
The lower bound is a consequence of Theorem~\ref{Completeproportion}, since
$$\max_{\substack{\|T\| = (d-1)n+1 \\ T \text{ strictly } d\text{-ary tree}}} \gamma(D,T) \geq \gamma(D,H_n^d) = \eta_d(D) + O(n^{-1}).$$
The upper bound follows directly from~\eqref{eq:arb_asymp}:
$$\max_{\substack{\|T\| = (d-1)n+1 \\ T \text{ strictly } d\text{-ary tree}}} \gamma(D,T)  \leq \max_{\substack{\|T\| = (d-1)n+1 \\ T\ d\text{-ary tree}}} \gamma(D,T) \leq I_d(D) + \frac{\|D\|(\|D\|-1)}{(d-1)n+1}\,.$$
\end{proof}

As mentioned in the introduction, this result gives support to the conjecture that~\eqref{eq:strict_speed} holds for arbitrary trees $D$ (in general, it has only been proven with an error term $O(n^{-1/2})$). As we have seen in this section, there are many examples for which the condition $I_d(T) = \eta_d(T)$ holds, such as all even trees.

\section{Further bounds}

Even when Theorem~\ref{upperonBalnce} does not yield the precise value of the inducibility, it often gives us very good bounds. In the case where $D$ has $d$ identical branches, the following theorem shows that it is at least ``almost sharp''.

\begin{theorem}\label{thm:equalbranches}
Let $d\geq 2$ be an arbitrary but fixed positive integer and $D$ a $d$-ary tree. Assume that $D$ has $d$ branches all of which are isomorphic to the same $d$-ary tree, say $D^{\prime}$. Then we have
$$ \frac{\|D\|!}{d^{\|D\|}}  \Bigg(\frac{I_d(D^{\prime})}{\|D^{\prime}\|!}\Bigg)^d \leq I_d(D)\leq \frac{\|D\|!}{d^{\|D\|}-d} \Bigg(\frac{I_d(D^{\prime})}{\|D^{\prime}\|!}\Bigg)^d \,.$$
\end{theorem}

\begin{proof}
The lower bound is a special case of \cite[Theorem 9]{AudaceStephanPaper1}, while the upper bound is a direct consequence of Theorem~\ref{upperonBalnce}.
\end{proof}

We conclude with a lower bound on the inducibility of a tree in terms of the inducibilities of its branches, of which the lower bound in the previous theorem is also a special case.

\begin{theorem}\label{FurAdd}
Let $D$ be a $d$-ary tree with branches $D_1,D_2,\ldots,D_d$ (some of which may be empty). The following inequality holds:
$$I_d(D) \geq \binom{\|D\|}{\|D_1\|,\|D_2\|,\ldots,\|D_d\|} \|D\|^{-\|D\|} \prod_{i=1}^d \|D_i\|^{\|D_i\|} \prod_{i=1}^d I_d(D_i)\,.$$
\end{theorem}

\begin{proof}
Let us write $k = \|D\|$ and $\ell_i = \|D_i\|$ as in previous proofs.
For each $D_i$, we can find a sequence of rooted trees $T_n^{(i)}$ such that $\|T_n^{(i)}\| = n$ and $\lim_{n \to \infty} \gamma(D_i,T_n^{(i)}) = I_d(D_i)$. Now define a new sequence of trees $T_n$ as follows:
\begin{itemize}
\item For each $i \in [d]$, take a copy of the tree $T_{\ell_i n}^{(i)}$, which has $\ell_i n$ leaves (if $\ell_i = 0$, this is the empty tree).
\item Add a new root, which is connected to the roots of all these trees by an edge.
\end{itemize}
Note that the tree $T_n$ has $\sum_{i=1}^d (\ell_i n) = kn$ leaves. If we take a leaf set in the $i$-th branch that induces a copy of $D_i$ for each $i$, then the union of all these leaf sets induces a copy of $D$. Therefore, we have
$$c(D,T_n) \geq \prod_{i=1}^d c(D_i,T_{\ell_i n}^{(i)})\,.$$
This also follows easily from~\eqref{The general recursion}. Now note that $c(D_i,T_{\ell_i n}^{(i)}) = \binom{\ell_i n}{\ell_i} \gamma(D_i,T_{\ell_i n}^{(i)})$ and $c(D,T_n) = \binom{kn}{k} \gamma(D,T_n)$. It follows that
$$\gamma(D,T_n) \geq \frac{\prod_{i=1}^d \binom{\ell_i n}{\ell_i}}{\binom{k n}{k}} \prod_{i=1}^d \gamma(D_i,T_{\ell_i n}^{(i)})\,.$$
As $n \to \infty$, the right side of this inequality tends to
$$\frac{\prod_{i=1}^d \frac{\ell_i^{\ell_i}}{\ell_i!}}{\frac{k^k}{k!}} \prod_{i=1}^d I_d(D_i) = \binom{k}{\ell_1,\ell_2,\ldots,\ell_d} k^{-k} \prod_{i=1}^d \ell_i^{\ell_i} \prod_{i=1}^d I_d(D_i)\,,$$
so it follows that
$$I_d(D) \geq \limsup_{n \to \infty} \gamma(D,T_n) \geq \binom{k}{\ell_1,\ell_2,\ldots,\ell_d} k^{-k} \prod_{i=1}^d \ell_i^{\ell_i} \prod_{i=1}^d I_d(D_i)\,,$$
which completes the proof.
\end{proof}

Let us illustrate the results of this section with two final examples, which are shown in Figure~\ref{fig:final_examples}. For the ternary tree $T_1$ on the left, Theorem~\ref{thm:equalbranches} yields $0.08535 \approx \frac{560}{6561} \leq I_3(T_1) \leq \frac{7}{82} \approx 0.08537$, giving us an excellent approximation. The lower bound provided by Theorem~\ref{Completeproportion} is much weaker in this case, as $\eta_3(T_1) = \frac{189}{5248} \approx 0.03601$. For the binary tree $T_2$ on the right, Theorem~\ref{FurAdd} yields $I_2(T_2) \geq \frac{80}{243} \approx 0.32922$, which is stronger than the lower bound $\eta_2(T_2) = \frac{60}{217} \approx 0.27650$. An upper bound can be obtained from Proposition~\ref{ind less prod ind}, which gives us $I_2(T_2) \leq \frac{15}{31} \approx 0.48387$. We do not know the precise value of the inducibility in either of the two cases, though.

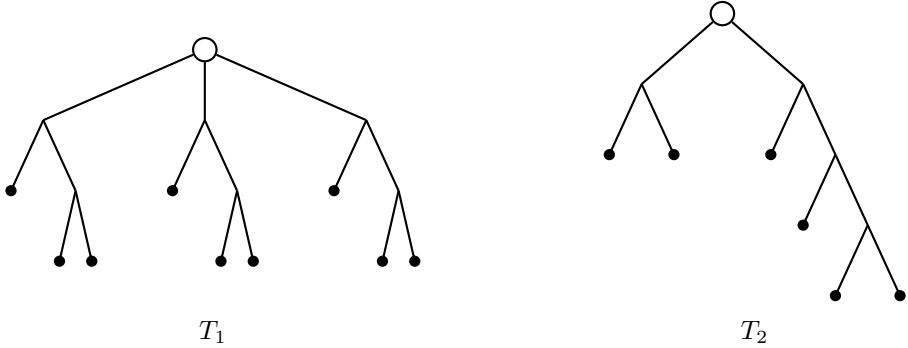
\begin{figure}[htbp]\centering 
\begin{subfigure}[b]{0.4\textwidth} \centering 
\begin{tikzpicture}[thick,level distance=11mm, scale=0.85, transform shape]
\tikzstyle{level 1}=[sibling distance=25mm]   
\tikzstyle{level 2}=[sibling distance=10mm]
\tikzstyle{level 3}=[sibling distance=5mm]
\node [circle,draw]{}
child {child {[fill] circle (2pt)}child {child {[fill] circle (2pt)}child {[fill] circle (2pt)}}}
child {child {[fill] circle (2pt)}child {child {[fill] circle (2pt)}child {[fill] circle (2pt)}}}
child {child {[fill] circle (2pt)}child {child {[fill] circle (2pt)}child {[fill] circle (2pt)}}};
\end{tikzpicture}\qquad \qquad \qquad
\caption{$T_1$}
\end{subfigure}
\begin{subfigure}[b]{0.5\textwidth} \centering 
\begin{tikzpicture}[thick,level distance=11mm, scale=0.85, transform shape]
\tikzstyle{level 1}=[sibling distance=25mm]   
\tikzstyle{level 2}=[sibling distance=10mm]
\tikzstyle{level 3}=[sibling distance=10mm]
\tikzstyle{level 4}=[sibling distance=10mm]
\node [circle,draw]{}
child {child {[fill] circle (2pt)}child {[fill] circle (2pt)}}
child {child {[fill] circle (2pt)}child {child{[fill] circle (2pt)}child{child {[fill] circle (2pt)}child {[fill] circle (2pt)}}}};
\end{tikzpicture}
\caption{$T_2$}
\end{subfigure}
\caption{Two final examples.}\label{fig:final_examples}
\end{figure}

\section*{Acknowledgment}

We would like to thank an anonymous referee for valuable comments which helped us to improve this paper.

\end{document}